\documentclass[12pt]{amsart}
\usepackage{amsfonts,amsmath,amssymb,amscd,amsthm,enumerate}
\usepackage{amsrefs,comment}
\usepackage{geometry}
\usepackage{xcolor}
\usepackage{adjustbox}
\usepackage[all]{xy}
\usepackage{stmaryrd}
\newtheorem{theorem}{Theorem}
\newtheorem{prop}[theorem]{Proposition}
\newtheorem{lem}[theorem]{Lemma}
\newtheorem{cor}[theorem]{Corollary}
\theoremstyle{definition}

\newtheorem{rem}[theorem]{Remark}
\newtheorem{mydef}[theorem]{Definition}
\newtheorem{example}[theorem]{Example}

\renewcommand{\epsilon}{\varepsilon}
\def\lbra{{[}\!{[}}
\def\rbra{{]}\!{]}}

\def\R{\mathbb{R}}

\def\T{\mathbb{T}}

\def\x{\mathbf{x}}

\def\y{\mathbf{y}}
\def\z{\mathbf{z}}

\def\<{\langle}
\def\>{\rangle}

\usepackage{multicol} %%%%%%%%%%%

\begin{document}
\title{A note on the shifted Courant-Nijenhuis torsion}
\author{Marco Aldi, Sergio Da Silva, and Daniele Grandini}
\begin{abstract} We characterize the vanishing of the shifted Courant-Nijenhuis torsion as the strongest tensorial integrability condition that can be imposed on a skew-symmetric endomorphism of the generalized tangent bundle.
\end{abstract}
% Keywords command
\providecommand{\keywords}[1]
{
  \small	
  \textbf{\textit{Keywords---}} #1
}

\keywords{Generalized geometry, Courant algebroids, polynomial structures}
\subjclass[2020]{53D18, 14P99}
\maketitle

\section{Introduction}

Let $M$ be a (paracompact) smooth manifold and let $\T M$ be its generalized tangent bundle (i.e.\ the direct sum of its tangent and cotangent bundles). A generalized almost complex structure on $M$ is a skew-symmetric (with respect to the tautological inner product on $\T M$) bundle endomorphism $J:\T M\to \T M$ such that $J^2+{\rm id}=0$. Generalized almost complex structures that satisfy the integrability condition 
\begin{equation}\label{eq:1}
0 = \mathcal T_J(\x,\y)= \lbra J(\x),J(\y)\rbra + J^2 (\lbra \x,\y\rbra ) - J(\lbra J(\x),\y\rbra+\lbra \x,J(\y)\rbra)
\end{equation}
for all smooth sections $\x,\y$ of $\T M$, where $\lbra\cdot,\cdot\rbra$ denotes the Courant-Dorfman bracket, are known as generalized complex structures and play a special role in Generalized Geometry \cite{Gualtieri11}. The map $\mathcal T_J$, known as the Courant-Nijenhus torsion of $J$, has the remarkable property of being tensorial in the sense that $\mathcal T_J(f\x,\y)=f\mathcal T_J(\x,\y)=\mathcal T_J(\x,f\y)$ whenever $f$ is a smooth function on $M$ and $\x,\y$ are smooth sections of $\T M$. As a result, it suffices to check the integrability condition \eqref{eq:1} on local frames. Moreover, since the Courant-Dorfman bracket is the derived bracket of the de Rham operator, any non-tensorial compatibility condition with an endomorphism (which is by definition $C^{\infty}(M)$-linear) of $\T M$ would necessarily impose non-trivial differential equations to be satisfied by all smooth functions on $M$.

Motivated by earlier analysis of the cubic case \cites{Vaisman08, PoonWade11} and classical work on higher degree structures on the tangent bundle \cites{GoldbergYano70, GoldbergPetridis73}, the notion of a generalized polynomial structure was introduced in \cite{AldiGrandini22}. A generalized polynomial structure extends the notion of a generalized almost complex structure by relaxing the $J^2+{\rm id}=0$ condition to $P(J)=0$ with $P$ an arbitrary polynomial in $\mathbb R[x]$. Many features of generalized almost complex structures (such as the decomposition of the complexification of $\T M$ into subspaces labeled by roots of $P$) carry through to arbitrary generalized polynomial structures. However the integrability condition \eqref{eq:1} does not straightforwardly apply to arbitrary polynomial structures. One explanation for this is that the Courant-Nijenhuis torsion is {\it not} tensorial for arbitrary generalized polynomial structures. On the other hand, remarkably, the so-called shifted Courant-Nijenhuis torsion 
\begin{equation}\label{eq:2}
    \mathcal S_J(\x,\y)=\mathcal T_J(J(\x),\y)+\mathcal T_J(\x,J(\y))
\end{equation}
is tensorial for {\it any} skew-symmetric endomorphism $J$ of $\T M$. Moreover, the integrability of almost complex structures is equivalent to the vanishing of their shifted Courant-Nijenhuis. In fact, as argued in \cite{AldiGrandini22}, the vanishing of the shifted Courant-Nijenhuis torsion is a sensible integrability condition for generic generalized polynomial structures (extra care must be taken when the polynomial $P$ has repeated roots).  

It is natural to ask whether the shifted Courant-Nijenhuis torsion is special in this regard and whether there are analogous tensorial expressions whose vanishing could give rise to exotic integrability conditions. In this paper we show that the shifted Courant-Nijenhuis torsion is the strongest possible tensorial condition expressed (in a precise way which will be specified below) in terms of the Courant-Dorfman bracket that can be uniformly imposed on all skew-symmetric endomorphisms of the generalized tangent bundle. Our main observation is that, using a construction of Kosmann-Schwarzbach \cite{Kosmann-Schwarzbach19}, ``Nijenhuis-like'' expressions can be represented by real polynomials in three variables and polynomials representing tensorial expressions form an ideal which, as we prove using methods of real algebraic geometry, is generated by the polynomial representing the shifted Courant-Nijenhuis torsion. This is the content of Theorem \ref{theorem: main}.

It seems plausible that our techniques yield results in a much more general context. We conclude this paper by listing some possible directions for further research such as symmetric endomorphisms (modeled after generalized Riemannian structures), and tensorial objects on more general algebroids. Since our primary goal is to provide additional context for our main result, we limit ourselves to a sketch of these open problems, leaving a systematic treatment for future work.

\section{Preliminaries}

In this section we will review the geometry of the generalized tangent bundle (referring the reader to \cite{Gualtieri11} for further details)  and introduce the basic vocabulary needed to state our main theorem. 

The generalized tangent bundle
$\T M= TM \oplus T^*M$ of a manifold $M$ has the following canonical structures:

\begin{enumerate}[i)]
\item the {\it anchor map}, defined as the projection map $\pi:\T M\rightarrow TM$; 
\item the {\it tautological inner product} $\langle \cdot,\cdot \rangle$, which is the nondegenerate, symmetric $C^\infty(M)$-bilinear map defined by
\begin{equation*}
    \langle X+\alpha, Y+\beta\rangle=\frac{1}{2}\left(\alpha(Y)+\beta(X)\right)
\end{equation*}
for all sections $X,Y\in \Gamma(TM)$ and $\alpha,\beta\in \Omega_M^1$;
\item the $\R$-bilinear map $\lbra\cdot ,\cdot \rbra:\Gamma(\T M)\otimes_{\R}\Gamma(\T M)\rightarrow \Gamma(\T M)$ such that
\begin{equation*}
\lbra X+\alpha, Y+\beta\rbra=[X,Y]+{\mathcal L}_X\beta-i_Yd\alpha,
\end{equation*}
for all $X,Y\in \Gamma(TM)$ and $\alpha,\beta\in \Omega_M^1$, called the {\it Courant-Dorfman bracket} on $M$.
\end{enumerate} 

\noindent Here $\mathcal{L}_X$ is the Lie derivative along the vector field $X$, and $i_Y$ is the interior derivative along $Y$. The data $(\T M,\pi, \langle\cdot ,\cdot \rangle, \lbra\cdot ,\cdot \rbra)$ defines a {\it Courant Algebroid} \cite{Gualtieri11}. In particular, the following property is satisfied:
\begin{equation}\label{eq:Courant}\pi(\x)\langle\y,\z\rangle=\langle\lbra\x,\y\rbra,\z\rangle+\langle\lbra\x,\z\rbra,\y\rangle=\langle\lbra\y,\z\rbra,\x\rangle+\langle\lbra\z,\y\rbra,\x\rangle\, ,
\end{equation} 

\noindent for all $\x,\y,\z\in \Gamma(\T M)$, where $\pi(\x)\langle\y,\z\rangle$ denotes the usual action of the vector field $\pi(\x)$ on the smooth function $\langle\y,\z\rangle$. It well-known that \eqref{eq:Courant} implies the Leibniz identities
\begin{equation}\label{eq:Leibniz}
\lbra\x,f\y\rbra=f\lbra\x,\y\rbra+\pi(\x)(f)\y, \qquad
\lbra f\x,\y\rbra=f\lbra\x,\y\rbra-\pi(\y)(f)\x+2\langle\x,\y\rangle df
\end{equation}
for all $\x,\y\in \Gamma(\T M)$ and $f\in C^\infty(M)$.
\begin{mydef}  Let $\Theta$ be the $C^\infty(M)$-module of $\R$-trilinear forms
\begin{equation}
\tau:\Gamma(\T M)\otimes_{\mathbb R} \Gamma(\T M)\otimes_{\mathbb R}\Gamma(\T M)\rightarrow C^\infty(M)\,.
\end{equation}
Distinguished, among all element of $\Theta$ is the {\it Courant element}  $\tau_C$ such that \begin{equation}
\tau_C(\x,\y,\z)=\langle\lbra\x,\y\rbra,\z \rangle
\end{equation}
for all $\x,\y,\z\in \Gamma(\T M)$ \cite{Courant90}.
Moreover, any vector bundle morphism $J:\T M\rightarrow \T M$ induces an action $\bullet_J$ of the polynomial ring $\R[x,y,z]$ on $\Theta$, uniquely defined by
\begin{align*}
x\bullet_J \tau (\x,\y,\z) & = \tau(J(\x),\y,\z),\\
y\bullet_J \tau (\x,\y,\z) & = \tau(\x,J(\y),\z),\\
z\bullet_J \tau (\x,\y,\z) & = \tau(\x,\y,J(\z)).
\end{align*}
\end{mydef}
\begin{rem}
The polynomial action $\bullet_J$ was first introduced in \cite{Kosmann-Schwarzbach19} (see also \cite{AldiGrandini22} for applications to polynomial structures). 
\end{rem}

\begin{example}\label{ex:3}
Let $J:\T M\rightarrow \T M$ be a skew-symmetric endomorphism. Then
\begin{equation}
(x+z)(y+z)\bullet_J\tau_C=\langle{\mathcal T}_J(\x,\y), \z\rangle
\end{equation}
where $\mathcal T_J$ is the {\it Courant-Nijenhuis torsion} of $J$ defined in Equation \eqref{eq:1}. A direct calculation shows that if $J$ is a generalized almost complex structure, then the $\mathbb R$-trilinear form $\langle{\mathcal T}_J(\x,\y), \z\rangle$ is in fact a tensor (i.e.\ $C^\infty(M)$-trilinear).
\end{example}

\begin{mydef} A {\it tensorial pair} is given by $(P, J)$ where $P=P(x,y,z)\in\R[x,y,z]$, and $J$ is a skew endomorphism  such that $P\bullet_J\tau_C$ is $C^\infty(M)$-trilinear. Moreover, a polynomial $P$ is called {\it tensorial} if $(P, J)$ is a tensorial pair for all skew endomorphisms $J:\T M\rightarrow \T M$ and any smooth manifold $M$.
\end{mydef}

\begin{example}
Example \ref{ex:3} can be now reframed as the statement that $((x+z)(y+z),J)$ is a tensorial pair whenever $J$ is an almost complex structure. Let us show that the Courant-Nijenhuis torsion is not in general tensorial by exhibiting a skew-symmetric endomorphism $J$ of $\T M$ such that the pair $((x+z)(y+z), J)$ is not a tensorial pair. Let $M$ be the three-dimensional Heisenberg Nilmanifold i.e.\ the quotient of the group of 3-by-3 upper-triangular real matrices with 1s along the diagonal by the subgroup of matrices with integer entries. Let us fix a global frame $X_1,X_2,X_3$ (so that $[X_i, X_j]=\epsilon_{ij}^kX_k$) and let us denote its corresponding coframe by $\alpha_1,\alpha_2,\alpha_3$. Consider the skew-symmetric endomorphism $J$ of $\T M$ such that $JX_1=X_2$, $JX_2=X_3+\alpha_3$, $JX_3=-\alpha_2$, $J\alpha_1=0$, $J\alpha_2=-\alpha_1$, and $J\alpha_3=-\alpha_2$.
Then, setting $\x=X_3+\alpha_3$ and $\y=X_3-\alpha_3$ and letting $f$ be an arbitrary smooth function on $M$ we obtain
\begin{equation}\label{eq:8}
{\mathcal T}_J(f\x,\y)-f\mathcal T_J(\x,\y) =-2J(\lbra -f\alpha_2, X_3\rbra)-2fJ(\lbra\alpha_2,X_3\rbra) =2X_3(f)\alpha_1\,,
\end{equation}
which vanishes only under the additional requirement $X_3(f)=0$. Hence the {\it Courant-Nijenhuis polynomial} $(x+z)(y+z)$ is not tensorial.

\end{example}

It follows from Equation \eqref{eq:8} that imposing the non-tensorial condition $\mathcal T_J(\x,\y)=0$ on all smooth sections $\x,\y$ of $\T M$ results in a differential equation $X_3(f)=0$ imposed on {\it all} smooth functions $f$ on $M$. This is a general feature of non-tensorial expressions written in terms of the Courant-Dorfman bracket, which is the derived bracket of the de Rham operator \cite{Kosmann-Schwarzbach04}. 

As another example, consider the non-tensorial condition $(x+y)\bullet_J\tau_C=0$ for some (necessarily $C^\infty(M)$-linear) vector bundle endomorphism $J$ of $\T M$. Suppose $\x,\y$ are smooth sections of $\T M$ satisfying this condition (i.e.\ $\lbra J(\x),\y\rbra+\lbra \x,J(\y)\rbra=0$) and let $f\in C^\infty(M)$. Then $(x+y)\bullet_J\tau_C=0$ implies 
\begin{equation}
0=\lbra J(\x),f\y\rbra+\lbra \x,J(f\y)\rbra = \pi(\x)(f)J(\y)+\pi(J(\x))(f)\y\, ,
\end{equation}
which is a non-trivial system of differential equations imposed on all smooth functions of $M$. We conclude that any sensible compatibility condition between the Courant-Dorfman bracket and a bundle endomorphism $J$ of $\T M$ is necessarily tensorial.

\begin{example}\label{ex:minimality}
Let $J:\T M\to \T M$ be a generalized polynomial structure such that $P(J)=0$ for some polynomial $P\in \R[t]$. An integrability condition (called {\it minimality} in \cite{AldiGrandini22}) for generalized polynomial structures can be stated as the vanishing of the trilinear form
\begin{equation}\label{eq:7}
    \mathcal C_J(\x,\y,\z)=\sum_{i=0}^N \sum_{i_1+i_2+i_3=i} a_i (-1)^i\binom{i}{i_1,i_2,i_3} \tau_C(J^{i_1}(\x),J^{i_2}(\y),J^{i_3}(\z))
\end{equation}
for all $\x,\y,\z\in \T M$, where $P(t)=a_0+a_1t+\cdots + a_Nt^N$. In the case in which $J$ is a generalized almost complex structure, \eqref{eq:7} is equivalent to \eqref{eq:1}. As it turns out, the $\R$-trilinear form $\mathcal{C}_J$ is actually $C^\infty(M)$-trilinear and equal to $Q\bullet_J\tau_C$
where $Q(x,y,z)=P(x+y+z)$. Hence $(Q,J)$ is a tensorial pair. 
\end{example}

\begin{prop}\label{prop:tens} Let $P\in \R[x,y,z]$, and let $J:\T M\rightarrow \T M$ be skew. The pair $(P,J)$ is tensorial 
if and only if the following conditions are satisfied:
\begin{align}
\sum_{i,j,t}(-1)^ja_{i,j,t-j}\pi\left(J^i(\x)\right)(f)J^{t}& =0\label{eq:9}\\ \sum_{i,j,t}\left((-1)^{j}a_{j,i,t-j}\pi\left(J^i(\y)\right)(f)J^{t}(\z)\right) & =  \sum_{i,j,t}\left((-1)^j    a_{j,t-j,i}\pi\left(J^i(\z)\right)(f)J^{t}(\y)\right)\label{eq:10}
\end{align}
for all $\x,\y,\z\in \Gamma(\T M)$ and $f\in C^\infty(M)$, where $P(x,y,z)=\sum_{i,j,k}a_{i,j,k}x^iy^jz^k$.
\end{prop}
\begin{proof} $P\bullet_J\tau_C$ is a tensor if and only if
\begin{equation}\label{eq:13}
\sum_{i,j,k}a_{i,j,k}\left(\langle\lbra J^i(\x),J^j(f\y)\rbra, J^k(\z)\rangle-f\langle\lbra J^i(\x),J^j(\y)\rbra, J^k(\z)\rangle\right)=0\,,
\end{equation}
and
\begin{equation}\label{eq:14}
\sum_{i,j,k}a_{i,j,k}\left(\langle\lbra J^i(f\x),J^j(\y)\rbra, J^k(\z)\rangle-f\langle\lbra J^i(\x),J^j(\y)\rbra, J^k(\z)\rangle\right)=0
\end{equation}
hold for all $\x,\y,\z\in \Gamma(\T M)$ and $f\in C^\infty(M)$.
Using \eqref{eq:Leibniz}, Equations \eqref{eq:13} and \eqref{eq:14} are, respectively, equivalent to
\begin{equation}\label{eq:15}
\sum_{i,j,k}a_{i,j,k}\pi(J^i(\x))(f)\langle J^j(\y), J^k(\z)\rangle=0
\end{equation}
and
\begin{equation}\label{eq:16}
\sum_{i,j,k}a_{i,j,k}\left(-\pi(J^j(\y))(f)\langle J^i(\x),J^k(\z)\rangle+\pi(J^k(\z))(f)\langle J^i(\x),J^j(\y)\rangle\right)=0\,.
\end{equation}
Using the skew-symmetry of $J$ and the nondegeneracy of the tautological inner product shows that \eqref{eq:15} is equivalent to \eqref{eq:9}. Similarly, \eqref{eq:16} is equivalent to \eqref{eq:10}. 
\end{proof}
\begin{theorem}\label{theo:unitens} Let $P=\sum_{i,j,k}a_{i,j,k}x^iy^jz^k \in \R[x,y,z]$. The following conditions are equivalent:
\begin{enumerate}[(i)]
    \item  $P$ is tensorial;
    \item  the following system of equations is satisfied:
\begin{equation*}\sum_{j}(-1)^ja_{i,j,t-j}=0,\quad \sum_{j}(-1)^ja_{j,t-j,i}=0, \quad \sum_{j}(-1)^ja_{t-j,i,j}=0\end{equation*}
for all non-negative integers $i,t$;
\item The polynomial $P$ vanishes on the following subvariety of $\R^{3}$:
\begin{equation*}
V=\left\{(x,y,z):y+z=0\right\}\cup \left\{(x,y,z):z+x=0\right\}\cup\left\{(x,y,z):x+y=0\right\}\,.
\end{equation*}
\end{enumerate}
\end{theorem}
\begin{proof} It is clear that (iii) implies (ii) and, by Proposition \ref{prop:tens}, (ii) implies (i). It remains to show that if (i) holds, then  the polynomials $A(u,v):=P(v,-u, u)$, $ B(u,v):=P(-u, v,u)$, and $C(u,v):=P(-u, u, v)$
vanish identically. Let $\lambda,\mu\in \R$. Since the tensoriality of $P$ does not depend on the manifold $M$, we can assume ${\rm dim}(M)\geq 2$ and construct a skew-symmetric endomorphism $J:\T M \rightarrow \T M$ preserving the tangent bundle (i.e.\ $J(TM)\subseteq TM$ such that locally there are linearly independent vector fields $X,Y\in \Gamma(TM)$ such that $J(X)=\lambda X$, and $J(Y)=\mu Y$). 

Under these assumptions, it follows from Proposition \ref{prop:tens} that
$A(\lambda,\mu)X(f)Y=0$ and $B(\mu,\lambda)X(f)Y=C(\lambda,\mu)Y(f)X$
for all $f\in C^\infty(M)$. Therefore,
$A(\lambda,\mu)=B(\mu,\lambda)=C(\lambda,\mu)=0$.
\end{proof}

\begin{example}
Consider the {\it shifted Courant-Nijenhuis polynomial}
\begin{equation}\label{eq:19}
S(x,y,z)=(x+y)(y+z)(z+x)=x^2y+xy^2+x^2z+xz^2+y^2z+yz^2+2xyz\,.
\end{equation}
By inspection, the equations in Theorem \ref{theo:unitens} (ii) hold and thus $S$ is tensorial. It follows that the {\it shifted Courant Nijenhuis torsion} \eqref{eq:2} is tensorial for any skew-symmetric endomorphism $J$ of $\T M$.
\end{example}

Theorem \ref{theo:unitens} implies that the subset $\mathcal I\subseteq \R[x,y,z]$ of all tensorial polynomials coincides with the ideal $I(V)$. With respect to the diagonal scaling, we obtain the decomposition $\mathcal I = \bigoplus_D \mathcal I^D$ into graded components. Moreover, $\mathcal I$ is invariant under cyclic permutations of the variables. For instance,  ${\mathcal I}^0=0$, as the equations in Theorem \ref{theo:unitens} (ii) simply reduce to $a_{000}=0$. When $D=1$, the Theorem \ref{theo:unitens} (ii) equations reduce to
\begin{equation*}
    \{a_{100}=0, a_{010}-a_{001}=0\ \  \mbox{ + c.p. }\}
\end{equation*}
where \lq\lq c.p.\rq\rq refers to equations obtained from the preceding equations by permuting the indices cyclically. Again, this implies ${\mathcal I}^1=0$. Similarly, for $D=2$ we obtain 
\begin{equation*}
\{a_{200}=0, a_{110}-a_{101}=0, a_{020}-a_{011}+a_{002}=0\ \  \mbox{ + c.p. }\}
\end{equation*}
which yields ${\mathcal I}^2=0$. When $D=3$, the equations reduce to 
\begin{equation*}
    \{a_{300}=0, a_{210}-a_{201}=0, a_{120}-a_{111}+a_{102}=0, a_{030}-a_{021}+a_{012}-a_{003}=0 \ \  \mbox{ + c.p. }\}
\end{equation*}
which can be further reduced to 
$$\{a_{300}=0, a_{210}=a_{201}=\frac{1}{2}a_{111} \ \ \mbox{ + c.p. }\}$$
Therefore, ${\mathcal I}^3$ is one-dimensional and is spanned by shifted Courant-Nijenhuis polynomial \eqref{eq:19}.
As we will show, all tensorial polynomials can be expressed in terms of $S(x,y,z)$.

\begin{mydef}
Let $P\in\R[x,y,z]$ and let $J:\T M\rightarrow \T M$ be a skew-symmetric endomorphism. The pair $(P,J)$ is called {\it polynomially tensorial} if the equations in Theorem \ref{theo:unitens} (ii)  hold and $P(J,0,0)=P(0,J,0)=P(0,0,J)$.
\end{mydef}

It is a straightforward to verify that polynomially tensorial pairs are tensorial pairs (using Proposition \ref{prop:tens}). Also, if $P$ is a tensorial polynomial, then $(P,J)$ is a polynomially tensorial pair for every skew-symmetric endomorphism $J$ of $\T M$, as $P$ does not contain monomials in one variable. The class of polynomially tensorial pairs generalizes the construction of the tensor ${\mathcal C}_J$ in \cite{AldiGrandini22} and in Example \ref{ex:minimality}.

\begin{prop}\label{prop:Alternating}
If $(P,J)$ is polynomially tensorial pair, then the tensor $P\bullet_J\tau_C$ is alternating.
\end{prop}
\begin{proof} For all $\x,\y,\z\in \Gamma(\T M)$, let $\theta(\x,\y,\z)=\pi(\x)\langle\y,\z\rangle$ so that 
\begin{equation}
\left(P\bullet_J\theta\right)(\x,\y,\z)=\sum_{i,j,t}(-1)^ja_{i,j,t-j}\pi\left(J^i(\x)\right)\langle\y, J^t(\z)\rangle\,.
\end{equation}
Since $(P,J)$ is polynomially tensorial, then $P\bullet_J\theta=0$. Using this and, (\ref{eq:Courant}), we obtain
\begin{equation}
P\bullet_J\tau_C (\x,\y,\z)-P\bullet_J\tau_C (\x,\z,\y)=P\bullet_J\theta(\x,\y,\z)=0 \end{equation}
and
\begin{equation}P\bullet_J\tau_C (\x,\y,\z)-P\bullet_J\tau_C (\y,\x,\z)=P\bullet_J\theta(\z,\x,\y)=0\,, \end{equation}
which concludes the proof.
\end{proof}

\section{The Main Result}
Let $\mathcal I\subseteq \R[x,y,z]$ denote the ideal of all tensorial polynomials. In the notation of Theorem \ref{theo:unitens}, as a consequence of the Real Nullstellensatz \cite[Theorem A.3]{Powers22}, we obtain
\begin{equation}
{\mathcal I}=I(V)=\sqrt[Re]{I_xI_yI_z}\,,
\end{equation}
where $\sqrt[Re]{I_xI_yI_z}$ is the real radical of $I_xI_yI_z$ and 
$I_x, I_y, I_z$ are the following ideals of 
$\mathbb{R}[x,y,z]:$
\begin{equation}
I_x=\langle y+z\rangle \quad
I_y=\langle  z+x\rangle \quad
I_z=\langle x+y\rangle\,.
\end{equation}

\noindent We will show that $\sqrt[Re]{I_xI_yI_z}=I_xI_yI_z$.

\begin{lem}\label{lem: intersection_realradical}
Let $I_x,I_y$ and $I_z$ be defined as above. Then $I_x\cap I_y\cap I_z$ is a real radical ideal.
\end{lem}
\begin{proof}
Since each polynomial $x+y, y+z$ and $z+x$ is  irreducible and each has a non-singular zero in $\mathbb{R}^3$, $I_x, I_y$ and $I_z$ are real radical ideals by \cite[Theorem A.4]{Powers22}. To prove the result, it suffices to show that the intersection of real radical ideals is again a real radical ideal. On one hand  $I\cap J \subseteq \sqrt[Re]{I\cap J}$. On the other hand $\sqrt[Re]{I\cap J} \subseteq \sqrt[Re]{I}$ and $\sqrt[Re]{I\cap J} \subseteq \sqrt[Re]{J}$, therefore $\sqrt[Re]{I\cap J} \subseteq \sqrt[Re]{I}\cap \sqrt[Re]{J}= I\cap J$, as required. 
\end{proof}

\begin{lem}\label{lem: product_intersection}
The ideal $I_xI_yI_z$ is radical. In particular, $I_xI_yI_z=I_x\cap I_y\cap I_z$. 
\end{lem}
\begin{proof}
In general, the intersection of finitely many principal ideals in a unique factorization domain is again a principal ideal (generated by the least common multiple of the irreducible factors). 

More specifically for $I_x$ and $I_y$, since $I_xI_y\subseteq I_x\cap I_y$, it suffices to show the reverse inequality. Let $f\in I_x\cap I_y$. Then $f\in I_x$, so $f=r(y+z)$ for some $r\in \mathbb{R}[x,y,z]$. Similarly, $f\in I_y$, so $f=s(z+x)$ for some $s\in \mathbb{R}[x,y,z]$. Then $r(y+z)=s(z+x)$. Then $z+x$ divides $r(y+z)$. Since $z+x$ is prime (equivalently irreducible in a UFD), $z+x$ must divide $(y+z)$ or $r$. Since it doesn't divide $y+z$, it must divide $r$, so $r=\tilde{r}(z+x)$. Then $f=\tilde{r}(z+x)(y+z)\in I_xI_y$, as required. 

A similar argument shows that $(I_x\cap I_y)\cap I_z = (I_xI_y)I_z$, proving the result.
\end{proof}

\begin{theorem}\label{theorem: main}
Let $\mathcal I\subseteq \R[x,y,z]$ denote the set of all tensorial polynomials. Then
\[\mathcal{I} =I_xI_y I_z.\]

\end{theorem}\label{thm:14}
\begin{proof}
By Theorem \ref{theo:unitens}, $\mathcal{I}=I(V)\subset \mathbb{R}[x,y,z]$ where $V = Z(I_xI_yI_z)$ (ie. the vanishing set of $I_xI_yI_z$). By the real Nullstellensatz \cite[Theorem A.3]{Powers22}, $I(V) = \sqrt[Re]{I_xI_yI_z}$. On the other hand, Lemmas \ref{lem: intersection_realradical} and \ref{lem: product_intersection} imply

\[\sqrt[Re]{I_xI_yI_z} = I_xI_yI_z\]
which yields the conclusion.
\end{proof}

\begin{cor}
The ideal $\mathcal{I}$ of tensorial polynomials is generated by the shifted Nijenhuis polynomial $S$.
\end{cor}

\begin{proof} Clearly $\langle S\rangle\subseteq {\mathcal I}$. On the other hand, suppose $P\in{\mathcal I}$. Since ${\mathcal I}=I_xI_yI_z$, we obtain $P=Q\cdot (x+y)(y+z)(x+z)$ for some $Q\in \R[x,y,z]$. In other words, $P\in \langle S\rangle$.

\end{proof}

\section{Generalizations}

\subsection{Symmetric endomorphisms} The skew-symmetry assumption on $J$ was used in the proof of Proposition \ref{prop:tens}, though not in an essential way. A similar analysis can be carried out if $J$ is a {\it symmetric endomorphisms} of the generalized tangent bundle i.e.\  one for which $\langle J(\x),\y\rangle = \langle \x, J(\y)\rangle$ for all $\x,\y\in \Gamma (\T M)$. An important class of examples of symmetric endomorphisms is given by \emph{generalized Riemannian metrics} which satisfy the additional condition $J^2={\rm id}$.  The definition of tensorial pairs, polynomially tensorial pairs, and tensorial polynomials can be transferred verbatim to the context of symmetric endomorphisms of the generalized tangent bundle. Here is the analogue of Theorem \ref{theo:unitens}:
\begin{theorem}\label{theo:unitensymm} Let $P=\sum_{i,j,k}a_{i,j,k}x^iy^jz^k$. The following conditions are equivalent:
\begin{enumerate}[(i)]
    \item  $P$ is tensorial (in the context of symmetric endomorphisms);
    \item  the following system of equations is satisfied:
\begin{equation*}
\sum_{j}a_{i,j,t-j}=0,\quad \sum_{j}a_{j,t-j,i}=0, \quad \sum_{j}a_{t-j,i,j}=0
\end{equation*}
for all non-negative integers $i,t$;
\item The polynomial $P$ vanishes on the following subvariety of $\R^{3}$:
$$V:=\left\{(x,y,z):y-z=0\right\}\cup \left\{(x,y,z):z-x=0\right\}\cup\left\{(x,y,z):x-y=0\right\}.$$
\end{enumerate}
\end{theorem}
Similar calculations show that the ideal $\mathcal{I}'$ of tensorial polynomials (in the symmetric setting) is generated by 
\begin{equation*}
S'(x,y,z)=(x-y)(y-z)(x-z)
 \end{equation*}
Note that when $J$ is a generalized Riemannian structure, we get $S'\bullet_J\tau_C=0$. In other words, this construction produces no nontrivial analogue of the (shifted) Courant-Nijenhuis tensor for generalized Riemannian structures.

\subsection{More general algebroids}
It is worthwhile to note that the computations in the paper do not make use of the explicit expressions of the tautological inner product or the Courant-Dorfman bracket. In fact, all results in the paper can be generalized verbatim by replacing $\T M$ with a vector bundle $A$ with the following data:

\vspace{1mm}

\begin{enumerate}
\item a vector bundle map $\pi:A\rightarrow TM$;
\item a symmetric, nondegenerate inner product $\langle\cdot  ,\cdot \rangle$ on $A$;
\item an $\R$-bilinear bracket $\lbra\cdot ,\cdot \rbra$ on $A$; 
\item $\pi$, $\langle\cdot  ,\cdot \rangle$, and $\lbra\cdot ,\cdot \rbra$
satisfy the identity in Equation  (\ref{eq:Courant}). 
\end{enumerate}

\vspace{1mm}

\noindent We propose the terminology \emph{Proto-Courant algebroid} for a vector bundle $A$ with the structure defined in (1) to (4) above, generalizing the notion of \emph{Pre-Courant algebroid} introduced in \cite{Vaisman2004TransitiveCA}. Furthermore, with the exception of Proposition \ref{prop:Alternating}, our results still hold true if $\pi$, $\langle\cdot ,\cdot \rangle$, and $\lbra\cdot ,\cdot \rbra$ satisfy only the Leibniz identities (\ref{eq:Leibniz}). These types of algebroids are known as \emph{local almost-Leibniz algebroids} \cite{PreLeibniz}.

\section*{Acknowledgements} This work is supported in part by VCU Quest Award ``Quantum Fields and Knots: An Integrative Approach.''

\end{document}